\documentclass[preprint,12pt]{elsarticle}

\usepackage{graphicx}
\usepackage[cp1251]{inputenc}
\usepackage{amssymb}
\usepackage{amsthm}
\usepackage{cmap}

\usepackage{amsmath}

\biboptions{sort&compress}

\newcommand{\sysn}{\left\{\begin{array}{rcl}}
\newcommand{\sysk}{\end{array}\right.}

\newtheorem{theorem}{Theorem}[section]

\theoremstyle{example}

\newtheorem{proposition}[theorem]{Proposition}
\theoremstyle{definition}
\newtheorem{definition}[theorem]{Definition}

\journal{...}

\begin{document}

\title{The projectively Hurewicz property is $t$-invariant}

\author{Alexander V. Osipov}

\address{Krasovskii Institute of Mathematics and Mechanics, \\ Ural Federal
 University, Yekaterinburg, Russia}

\ead{OAB@list.ru}

\begin{abstract}
A space $X$ is {\it projectively Hurewicz} provided every
separable metrizable continuous image of $X$ is Hurewicz.

In this paper we prove that the projectively Hurewicz property is
$t$-invariant, i.e., if $C_p(X)$ is homeomorphic to $C_p(Y)$ and
$X$ is projectively Hurewicz, then $Y$ is projectively Hurewicz,
too.

\end{abstract}


\maketitle 


\section{Introduction}

Let $\mathcal{P}$ be a topological property. A.V. Arhangel'skii
calls $X$ {\it projectively $\mathcal{P}$} if every second
countable continuous image of $X$ is $\mathcal{P}$
\cite{arch3,arh}. The projective selection principles were
introduced and first time considered in \cite{koc}. Lj.D.R.
Ko\v{c}inac characterized the classical covering properties of
Menger, Rothberger, Hurewicz and Gerlits-Nagy in term of
continuous images in $\mathbb{R}^{\omega}$. Characterizations of
the classical covering properties in terms a selection principle
restricted to countable covers by cozero sets are given in
\cite{bcm}. In \cite{os10,os19} we obtained the functional
characterizations of all projective versions of the selection
properties in the Scheepers Diagram.

Let us recall that a topological space is {\it Hurewicz} if for
every sequence $(\mathcal{U}_n: n\in \mathbb{N})$ of open covers
of $X$, there is a sequence $(\mathcal{V}_n: n\in \mathbb{N})$
such that for every $n$, $\mathcal{V}_n$ is a finite subfamily of
$\mathcal{U}_n$ and every point of $X$ is contained in $\bigcup
\mathcal{V}_n$ for all but finitely many n's.

Recall that if $C_p(X)$ and $C_p(Y)$ are homeomorphic (linearly
homeomorphic, uniform homeomorphic), we say that the spaces $X$
and $Y$ are {\it $t$-equivalent ($l$-equivalent, $u$-equivalent)}.
The properties preserved by $t$-equivalence ($l$-equivalence,
$u$-equivalence) we call {\it $t$-invariant ($l$-invariant,
$u$-invariant)} \cite{arh2}.

\medskip
The following interesting results were obtained:

$\bullet$ (Lj.D.R. Ko\v{c}inac)  A space is Hurewicz if and only
if it is Lindel\"{o}f and projectively Hurewicz \cite{koc}.

$\bullet$ (L. Zdomskyy)  The Hurewicz property is $l$-invariant
(Corollary 7 in \cite{zd}).

$\bullet$ (N.V. Velichko) The Lindel\"{o}f property is
$l$-invariant \cite{vel}.

$\bullet$ (M. Krupski)  The projectively  Hurewicz property is
$l$-invariant (Theorem 1.5 in \cite{kr}).

In this paper we prove that the projectively  Hurewicz property is
$t$-invariant.

\section{Main definitions and notation}

Throughout this paper, all spaces are assumed to be Tychonoff. The
set of positive integers is denoted by $\mathbb{N}$. Let
$\mathbb{R}$ be the real line, we put $\mathbb{I}=[0,1]\subset
\mathbb{R}$, and let $\mathbb{Q}$ be the rational numbers. For a
space $X$, we denote by $C_p(X)$ the space of all real-valued
continuous functions on $X$ with the topology of pointwise
convergence. The symbol $\bf{0}$ stands for the constant function
to $0$. Since $C_p(X)$ is a homogenous space we may always
consider the point $\bf{0}$ when studying local properties of this
space.

A basic open neighborhood of $\bf{0}$ is of the form $[F,
(-\epsilon, \epsilon)]=\{f\in C(X): f(F)\subset (-\epsilon,
\epsilon)\}$, where $F$ is a finite subset of $X$ and
$\epsilon>0$.

 We recall that a subset of $X$ that is the
 complete preimage of zero for a certain function from~$C(X)$ is called a zero-set.
A subset $O\subseteq X$  is called  a cozero-set (or functionally
open) of $X$ if $X\setminus O$ is a zero-set.

\medskip



Many topological properties are characterized in terms
 of the following classical selection principles.
 Let $\mathcal{A}$ and $\mathcal{B}$ be sets consisting of
families of subsets of an infinite set $X$. Then:


$S_{fin}(\mathcal{A},\mathcal{B})$ is the selection hypothesis:
for each sequence $(A_{n}: n\in \mathbb{N})$ of elements of
$\mathcal{A}$ there is a sequence $(B_{n}: n\in \mathbb{N})$ of
finite sets such that for each $n$, $B_{n}\subseteq A_{n}$, and
$\bigcup_{n\in\mathbb{N}}B_{n}\in\mathcal{B}$.

$U_{fin}(\mathcal{A},\mathcal{B})$ is the selection hypothesis:
whenever $\mathcal{U}_1$, $\mathcal{U}_2, ... \in \mathcal{A}$ and
none contains a finite subcover, there are finite sets
$\mathcal{F}_n\subseteq \mathcal{U}_n$, $n\in \mathbb{N}$, such
that $\{\bigcup \mathcal{F}_n : n\in \mathbb{N}\}\in \mathcal{B}$.



\bigskip


\medskip
In this paper, by a cover we mean a nontrivial one, that is,
$\mathcal{U}$ is a cover of $X$ if $X=\bigcup \mathcal{U}$ and
$X\notin \mathcal{U}$.

 An open cover $\mathcal{U}$ of a space $X$ is:

 $\bullet$ an {\it $\omega$-cover} if every finite subset of $X$ is contained in a
 member of $\mathcal{U}$.

$\bullet$ a {\it $\gamma$-cover} if it is infinite and each $x\in
X$ belongs to all but finitely many elements of $\mathcal{U}$.

$\bullet$ {\it $\gamma_F$-shrinkable} if $\mathcal{U}$ is a
cozero $\gamma$-cover and there exists a $\gamma$-cover $\{F_U :
U\in \mathcal{U}\}$ of zero-sets of $X$ with $F_U\subset U$ for
every $U\in \mathcal{U}$.

$\bullet$ {\it $\omega$-groupable}  if there is a partition of the cover into finite parts such
that for each finite set $F\subseteq X$ and all but finitely many parts $\mathcal{P}$ of
the partition, there is a set $U\in \mathcal{P}$ with $F\subseteq U$ \cite{kosch}.

\bigskip
For a topological space $X$ we denote:

$\bullet$ $\mathcal{O}$ --- the family of all open covers of $X$;

$\bullet$ $\mathcal{O}^{\omega}_{cz}$ --- the family of all
countable cozero covers of the space $X$;

$\bullet$ $\Gamma$ --- the family of all open $\gamma$-covers of
the space $X$;

$\bullet$ $\Gamma_{cz}$ --- the family of all cozero
$\gamma$-covers of the space $X$;


$\bullet$ $\Omega^{gr}$ --- the family of open $\omega$-groupable
covers of the space $X$;

$\bullet$ $\Gamma_F$ --- the family of all cozero
$\gamma_F$-shrinkable covers of the space $X$.

\medskip
Since any infinite part of the $\gamma$-cover is also a $\gamma$-cover, we further assume that all $\gamma_F$-shrinkable covers are countable.

\medskip

Let us recall that a topological space $X$ is {\it Hurewicz} if
$X$ has the property $U_{fin}(\mathcal{O},\Gamma)$.

\section{The projectively Hurewicz property}

A space $X$ is {\it projectively Hurewicz} provided every
separable metrizable continuous image of $X$ is Hurewicz.

In (\cite{bcm}, Theorem 30), M. Bonanzinga, F. Cammaroto, M.
 Matveev proved

\begin{theorem}\label{th33} The following
conditions are equivalent for a space $X$:

\begin{enumerate}

\item $X$ is projectively $U_{fin}(\mathcal{O},\Gamma)$
$[projectively Hurewicz]$;

\item Every Lindel$\ddot{o}$f continuous image of $X$ is Hurewicz;

\item for every continuous mapping $f: X \mapsto
\mathbb{R}^{\omega}$, $f(X)$ is Hurewicz;

\item for every continuous mapping $f: X \mapsto
\mathbb{R}^{\omega}$, $f(X)$ is bounded;

\item  $X$ satisfies $U_{fin}(\mathcal{O}^{\omega}_{cz},\Gamma)$.

\end{enumerate}

\end{theorem}

\begin{proposition}(Proposition 31 in \cite{bcm})\label{pro4}

\begin{enumerate}

\item Every $\sigma$-pseudocompact space is projectively Hurewicz.

\item Every space of cardinality less than $\mathfrak{b}$ is
projectively Hurewicz.

\item The projectively Hurewicz property is preserved by
continuous images, by countably unions, by $C^*$-embedded
zero-sets, and by cozero sets.

\end{enumerate}

\end{proposition}

\begin{definition} Let $\mathcal{S}=\{S_n: n\in \mathbb{N}\}$ be a family of
subsets of a space $X$ and $x\in X$. Then $\mathcal{S}$ {\it
weakly converges} to $x$ if for every neighborhood $W$ of $x$ there
is a sequence $(s_n: n\in \mathbb{N})$ such that $s_n\in S_n$ for
each $n\in \mathbb{N}$ and there is $n'$ such that $s_n\in W$ for
each $n>n'$.

\end{definition}

Let us recall that a subset $A$ of $X$ {\it converges} to $x$ if $A$ is infinite, $x\notin A$, and for each
neighborhood $U$ of $x$, $A\setminus U$ is finite. We write $x=\lim A$ if $A$ converges
to $x$. Consider the
following collections:

$\bullet$ $\Gamma_x=\{A\subseteq X : x=\lim A\}$;

$\bullet$ $w\Gamma_{x}=$ the family of all subsets of $X$ admitting a partition $\mathcal{S}=\{S_n: n\in \mathbb{N}\}$  such that for every $n$ the set $S_n$ is finite and  $\mathcal{S}$  weakly converges to $x$.

\begin{theorem} The following
conditions are equivalent for a space $X$:

\begin{enumerate}

\item $C_p(X)$ satisfies $S_{fin}(\Gamma_{{\bf 0}},
w\Gamma_{{\bf0}})$;

\item  $X$ satisfies $S_{fin}(\Gamma_{F}, \Omega^{gr})$;

\item $X$ satisfies $U_{fin}(\Gamma_{F}, \Gamma)$;

\item $X$ satisfies $U_{fin}(\mathcal{O}^{\omega}_{cz},\Gamma)$;

\item $X$ is projectively Hurewicz.

\end{enumerate}

\end{theorem}

\begin{proof}

$(3)\Leftrightarrow(2)$.
By Theorem 3.4 in \cite{ts1}, the equality $U_{fin}(\mathcal{O}, \Gamma)=S_{fin}(\Gamma, \Omega^{gr})$ is true in the class of metric separable spaces. Let $X$ satisfies $U_{fin}(\Gamma_{F}, \Gamma)$. By Theorem 5.4 in \cite{os19} and Theorem \ref{th33},
$U_{fin}(\Gamma_{F}, \Gamma)=U_{fin}(\mathcal{O}^{\omega}_{cz},\Gamma)$, i.e., $X$ is projectively Hurewicz.

Let $(\mathcal{U}_n : n\in \mathbb{N})$ be a sequence of countable $\gamma_F$-shrinkable covers of $X$. For every $n\in \mathbb{N}$ and every $U\in \mathcal{U}_n$ fix a continuous function $f_U:X \rightarrow \mathbb{R}$ such that $U=f^{-1}_U[\mathbb{R}\setminus\{0\}]$. Put
$h=\prod \{f_U : U\in \mathcal{U}_n, n\in \mathbb{N}\}$. Then $h$ is a continuous mapping from $X$ onto $h(X)\subseteq \mathbb{R}^{\omega}$, thus $h(X)$ satisfies $U_{fin}(\mathcal{O}, \Gamma)=S_{fin}(\Gamma, \Omega^{gr})$.
Let $h(\mathcal{U}_n)=\{h(U): U\in \mathcal{U}_n\}$.
Since $(h(\mathcal{U}_n) : n\in \mathbb{N})$ be a sequence of $\gamma$-covers of $h(X)$ we get (2). Since a continuous metrizable image of a space satisfying the property $S_{fin}(\Gamma_{F}, \Omega^{gr})$ is a space with this property and $S_{fin}(\Gamma_{F}, \Omega^{gr})=S_{fin}(\Gamma, \Omega^{gr})$ for metrizable spaces, the implication $(2)\Rightarrow(3)$ is proved similarly.

$(4)\Leftrightarrow(5)$. By Theorem \ref{th33}.

$(5)\Rightarrow(3)$. By Theorem 5.4 in \cite{os19} (or Theorem 4.1 in \cite{os10}).

$(3)\Rightarrow(4)$. Let $(\mathcal{U}_n : n\in \mathbb{N})$ be a
sequence of countable cozero covers of $X$. Enumerate
$\mathcal{U}_n=\{U^n_m: m\in \mathbb{N}\}$.

For $n,m\in \mathbb{N}$, fix a continuous function $f_{n,m}: X\rightarrow [0,1]$ that witnesses $U^n_m$ being cozero, i.e. $f^{-1}(0,1]=U^n_m$. For every $n,m,i\in \mathbb{N}$, let us define

$W^n_{m,i}=f^{-1}_{n,m}(\frac{1}{i+1},1]$ and $H^n_{m,i}=f^{-1}_{n,m}[\frac{1}{i+1},1]$.

Clearly, the set $W^n_{m,i}$ is cozero and $H^n_{m,i}$ is a zero-set. Note that

$W^n_{m,i}\subseteq H^n_{m,i}\subseteq W^n_{m,i+1}\subseteq U^n_m$ and $U^n_m=\bigcup\limits_{i=1}^{\infty} W^n_{m,i}.$

For $k\in\mathbb{N}$, write $W^n_k=\bigcup \{W^n_{m,i+1}: i,m\leq k\}$ and let $\mathcal{W}_n=\{W^n_1, W^n_2,...\}$. Observe that $\mathcal{W}_n\in \Gamma_F$ because $H^n_k=\bigcup\{H^n_{m,i}: i,m\leq k\}$ is a zero-set contained in $W^n_k$. Moreover the family $\{H^n_k: k\in \mathbb{N}\}$ is a $\gamma$-cover of $X$ since one readily checks that the family $\{\bigcup\{W^n_{m,i}: i,m\leq k\}: k\in \mathbb{N}\}$ is a $\gamma$-cover and $\bigcup\{W^n_{m,i}: i,m\leq k\}\subseteq H^n_k$.
Now apply the property $U_{fin}(\Gamma_F,\Gamma)$ to the sequence $(\mathcal{W}_n: n\in \mathbb{N})$ together with the fact that $\mathcal{W}_n$ is a finer cover that $\mathcal{U}_n$ for all $n$.

$(1)\Rightarrow(2)$.  Let $\{\mathcal{V}_i: i\in \mathbb{N}\}\in
[\Gamma_{F}]^{\omega}$. Note that we assume that all $\gamma_F$-shrinkable covers are countable.

 Since $\mathcal{V}_i=\{V_{i,j}: j\in
\mathbb{N}\}\in \Gamma_{F}$, there is $\{F_{i,j}: j\in
\mathbb{N}\}\in \Gamma$ such that $F_{i,j}$ is a zero-set in $X$
and $F_{i,j}\subset V_{i,j}\in \mathcal{V}_i$ for each $j\in
\mathbb{N}$. Let $T_i=\{f_{i,j}\in C_p(X): f_{i,j}(F_{i,j})=0$ and
$f_{i,j}(X\setminus V_{i,j})=1$ for each $i,j\in \mathbb{N}\}$.
Since $\{F_{i,j}:j\in\mathbb{N}\}$ is a $\gamma$-cover, we have $\lim\limits_{j\rightarrow \infty}T_i={\bf 0}$ for
each $i\in \mathbb{N}$. By (1), there are finite subsets $T'_i$ of
$T_i$ and a partition of the set $\bigcup T'_i$ into finite parts
such that for each neighborhood $O=[K,(-\epsilon,\epsilon)]$ of
the function {\bf 0} where $K$ is a finite subset of $X$ and $\epsilon>0$, and all but finitely many parts $\mathcal{P}$
of the partition, there is a function $g\in \mathcal{P}$ with
$g\in O$.

Let $\mathcal{P}=\{\{g_{l,1},...,g_{l,k_l}\}: l\in \mathbb{N}\}$.
Since $g_{l,m}=f_{i_s,j_s}$ for some $i_s,j_s\in \mathbb{N}$, we
can consider $Q=\{V_{l,m}: V_{l,m}=V_{i_s,j_s}$, $f_{i_s,j_s}(X\setminus
V_{i_s,j_s})=1$,
 $f_{i_s,j_s}=g_{l,m}$, $l\in \mathbb{N}\}$. Then $Q$
has a partition $\mathcal{Q}=\{\{V_{l,1},...,V_{l,k_l}\}: l\in\mathbb{N}\}$
and, for any finite subset $K$ of $X$  all but finitely many parts $\mathcal{Q}$
of the partition, there is $V_{l,k}$ with $K\subseteq V_{l,k}$. Thus, $Q\in \Omega^{gr}$.

$(2)\Rightarrow(1)$. Let $T_i\in\Gamma_{\bf 0}$ for each $i\in \mathbb{N}$.
By passing to a countable infinite subset, we can without loss of generality
assume that each $T_i$ is countable. Enumerate $T_i=\{f_{i,j}\in C_p(X): j\in
\mathbb{N}\}$.

For $i,j$ define $V_{i,j}=f^{-1}_{i,j}((-\frac{1}{i},\frac{1}{i}))$  (we can without loss of generality
assume that each $V_{i,j}$ is non-empty), and let $\mathcal{V}_i=\{V_{i,j}: j\in \mathbb{N}\}$.

 Note that $V_{i,j}$ is a cozero-set in $X$ for each $i,j\in \mathbb{N}$.

Thus we have a mapping $\Phi: \bigcup \mathcal{V}_i \rightarrow \bigcup T_i$ such that $\Phi(V_{i,j})=f_{i,j}$ for $i,j\in \mathbb{N}$.

 Since $\lim\limits_{j\rightarrow \infty}T_i={\bf 0}$, for any finite subset $F$ of $X$ and
$\epsilon>0$ (we can assume that $\epsilon<\frac{1}{i}$), there is $j'\in
\mathbb{N}$ such that $f_{i,j}\in [F,(-\epsilon,
\epsilon)]$ for each $j>j'$. Thus, $F\subset V_{i,j}$
for each $j>j'$. Thus,  $\mathcal{V}_i\in \Gamma_{cz}$.

For $i,j$ define $F_{i,j}=f^{-1}_{i,j}([-\frac{1}{i+1},\frac{1}{i+1}])$, and let $\mathcal{F}_i=\{F_{i,j}: j\in \mathbb{N}\}$.

 Then $F_{i,j}\subset V_{i,j}$
for each $j\in \mathbb{N}$ and $\mathcal{F}_i\in \Gamma$. Note
also that $F_{i,j}$ is a zero-set and $V_{i,j}$ is a cozero-set
in $X$ for each $j\in \mathbb{N}$. It follows that
$\mathcal{V}_i\in \Gamma_{F}$.

By (2), there are finite subsets $D_i\subset \mathcal{V}_i$ for
each $i\in \mathbb{N}$ such that $\bigcup D_i$ is a cozero
$\omega$-groupable cover of the space $X$.

Let
$\mathcal{P}=\{\mathcal{P}_k : k\in \mathbb{N}\}$ be a partition of the cover $\bigcup D_i$ into finite parts such that for each finite set
$F\subset X$ and all but finitely many parts $\{\mathcal{P}_k: k\in \mathbb{N}\}$  of the partition, there is a set $V_{i(k),j(k)}\in \mathcal{P}_k$  with
$F\subset V_{i(k),j(k)}$.

For each $k$ define $S_k=\{f_V:  \Phi(V)=f_V, V\in \mathcal{P}_k\}$. The family $\mathcal{S}=\{S_k : k\in \mathbb{N}\}$ is a partition  of $\bigcup \{f_{i,j}: V_{i,j}\in D_i, i\in \mathbb{N}\}$. Then, for each finite
set $F\subset X$ and $\epsilon>0$, and all but finitely many parts  of
the partition $\mathcal{S}$, there is a function $f_{i(k),j(k)}\in
\mathcal{S}_k$ with $f_{i(k),j(k)}\in [F,(-\epsilon,\epsilon)]$.
Thus, $\bigcup \{f_{i,j}: f_{i,j}\in T_i, V_{i,j}\in D_i, i\in \mathbb{N}\}\in  w\Gamma_{\bf
0}$ and  $C_p(X)$ satisfies $S_{fin}(\Gamma_{{\bf 0}}, w\Gamma_{\bf
0})$.
\end{proof}

Note that the property $S_{fin}(\Gamma_{x}, w\Gamma_x)$ is a {\it
topological} property. Thus, if $C_p(X)$ is homeomorphic to
$C_p(Y)$ and $C_p(X)$ satisfies $S_{fin}(\Gamma_{\bf 0}, w\Gamma_{\bf 0})$, then $C_p(Y)$ satisfies
$S_{fin}(\Gamma_{g}, w\Gamma_g)$ for each $g\in C_p(Y)$.

\begin{theorem} Suppose that $C_p(X)$ and $C_p(Y)$ are
homeomorphic. Then $X$ has the projectively Hurewicz property if
and only if $Y$ has the projectively Hurewicz property.
\end{theorem}


\medskip
{\bf Question}. Let $\mathcal{P}\in\{Menger, Rothberger,
Scheepers, S_1(\Gamma, \mathcal{O})\}$. Will the projectively
$\mathcal{P}$ property be $t$-invariant ?

\medskip

\medskip

{\bf Conjecture}. The projectively Scheepers Diagram is
$t$-invariant, i.e., each projectively selection property in the
Scheepers Diagram is $t$-invariant.

\medskip
If the conjecture is true, then, applying Velichko's result, the
Scheepers Diagram is $l$-invariant.

\medskip

{\bf Acknowledgements.} The author would like to thank the referee for careful reading and
valuable suggestions.


\bibliographystyle{model1a-num-names}
\bibliography{<your-bib-database>}

\end{document}